\documentclass[11pt, amssymb,amsfonts]{amsart}
\usepackage{amsrefs}
\usepackage{float}
\makeatletter
\def\l@section{\@tocline{1}{10pt plus0pt}{0pt}{}{\bfseries}}

\def\@tocline#1#2#3#4#5#6#7{\relax
    \ifnum #1>-1
  \ifnum #1>\c@tocdepth 
  \else
    \par \addpenalty\@secpenalty
    \begingroup \hyphenpenalty\@M
    \@ifempty{#4}{%
      \@tempdima\csname r@tocindent\number#1\endcsname\relax
    }{%
      \@tempdima#4\relax
    }%
    \parindent\z@ \leftskip#3\relax \advance\leftskip\@tempdima\relax
    \rightskip\@pnumwidth plus4em \parfillskip-\@pnumwidth
    #5\leavevmode\hskip-\@tempdima #6\nobreak\relax
    \hfil\hbox to\@pnumwidth{\@tocpagenum{#7}}\par
    \nobreak
    \endgroup
  \fi
\fi}

\makeatother
\usepackage{algorithm}
\usepackage{algorithmic}


\usepackage{hyperref}
\usepackage{a4wide}
\usepackage[a4paper, margin=2.3cm]{geometry}
\usepackage{amsmath,amssymb,amsthm}
\usepackage{color}
\usepackage{parskip}
\usepackage{graphicx}
\usepackage{hyperref}
\usepackage{parskip}
\usepackage{setspace}
\usepackage{amsmath}
\usepackage{mathtools}
\usepackage{amssymb}
\usepackage{cases}
\usepackage{algorithm}
\usepackage{algorithmic} 
\usepackage{todonotes}
\usepackage{stmaryrd}
\usepackage{IEEEtrantools}
\usepackage[nocompress]{cite}
\usepackage{tikz}
\usetikzlibrary{patterns,snakes}
 at 10 true pt

\def\be#1{ \begin{equation}\label{#1} }

\def\bas{\begin{align*}}
\def\eas{\end{align*}}
\def\bi{\begin{itemize}}
\def\ei{\end{itemize}}

\def\emph#1{{\it #1}}
\def\textbf#1{{\bf #1}}

\theoremstyle{plain}
 \theoremstyle{plain}
  \newtheorem{theorem}{Theorem}
  \numberwithin{theorem}{section}

  \newtheorem{fact}[subsection]{Fact}
  \newtheorem{lemma}{Lemma}
  
  \numberwithin{notation}{section}
  \numberwithin{lemma}{section}
  
   \numberwithin{corollary}{section}

\theoremstyle{remark}
  \newtheorem{remark}{Remark}
  \numberwithin{remark}{section}

\theoremstyle{definition}
  
\numberwithin{definition}{section}

\DeclareUnicodeCharacter{2113}{\ensuremath{\ell}}
\DeclareUnicodeCharacter{2082}{\ensuremath{\infty}}
\DeclareUnicodeCharacter{221E}{\ensuremath{\infty}}
\begin{document}
\include{psfig}

\setcounter{tocdepth}{2}


\title{A short proof of Kahn-Kalai conjecture}
\pagenumbering{arabic}

\author{Phuc Tran, Van Vu }
\thanks{Department of Mathematics, Yale University, 219 Prospect Ave, New Haven, Connecticut 06511, USA \\
  \textit{Email:} \texttt{phuc.tran@yale.edu, van.vu@yale.edu}.}
\date{}
\maketitle


\begin{abstract}
  In a recent paper,  Park and Pham famously proved Kahn-Kalai conjecture.
In this note, we simplify their proof, using an induction to replace the original analysis. This 
 reduces the proof to one page and from the argument it is also easy to read that one can set the constant $K$ in the conjecture to $\approx 3.998$, which 
 could be the best value under the current method.  Our argument also applies to the $\epsilon$-version of the Park-Pham result, studied by Bell. 
\end{abstract}

\section{Introduction}

Let $X$ be a set of $N$ atoms (elements) and $0  \le p \le 1$. The measure $\mu_p$ is defined on the subsets of $X$ by $\mu_p (S)= p^{|S|} (1-p) ^{N- |S| }$. 
This is the measure generated by choosing each atom from  $X$ independently with probability $p$. 
For a family $F$ of subsets, let $\mu_p (F) :=\sum_{S \in F} \mu_p (S)$. Furthermore, let $ \textbf{E}_p (F )= \sum_{S \in F} p^{| S|}$ be the expectation for the number of elements of 
$F$ in the set of chosen atoms. 

A family $G$ of subsets is  \textit{increasing}  if $ A \in G$ and $A \subset B$ then $B \in G$.
Given a family $F$, let $\langle F \rangle $ be the collection of subsets of $X$ which contain some element of $F$, namely 
$ \langle F \rangle := \{ T:   T \supset S, S \in F \} . $ 
It is trivial that $\langle F \rangle$ is an increasing family for any $F$. We say that $G$ {\it covers}  $F$ if $F \subset \langle G \rangle$. 

Let  $p_c (F)$ be the value where $\mu_{p_c} (\langle F \rangle ) =1/2$. Furthermore,  for a family $G$, let  $p_E (G)$ be the value where $\textbf{E}_{p_E}  (G) =1/2$. Now let 
$q(F) := \max \left\lbrace p_E (G) \mid G \,\, \text{covers} \,\,  F \right\rbrace $. It is clear that 
\begin{equation}
 p_E (F) \le q(F) \le p_c (F) . 
\end{equation}
Finally, we say that $F$ is $l$-bounded if its elements have size at most $l$.

\vskip2mm

{\noindent \it Example.} Let  $X$ be the set of all edges of the complete graph $K_n$ on a set $V$ of $n$ vertices. Thus, $|X|= N = {n \choose 2}$ and each subset $S \subset X$ corresponds to a graph on $V$. For  $0 \leq p \leq 1$,  $\mu_p$ is the measure of the random graph $G(n,p)$. Let $F$ be the collection of Hamiltonian cycles on $V$. Then $F$ is  $n$-bounded and $|F|= (n-1)!$. It follows that   $\textbf{E}_p (F) = \frac{(n-1)! p^n}{2}$; this is the expectation of the number of Hamiltonian cycles in $G(n,p)$. By definition,  $p_c (F)$ is the critical value where  $G(n,p)$ contains a Hamiltonian cycle with probability $1/2$. 
In random graph theory, $p_c(F)$ is referred to as the {\it threshold } for the appearance of a Hamiltonian cycle.

 By Stirling's formula,  it is easy to show that $q(F) \ge  p_E (F)= (1+o(1)) \frac{e}{n} $. The computation of $p_c (F)$ is harder, and classical theorems in random graph theory show that  
 $$p_c (F)=  (1+o(1)) \frac{\ln n}{n}= (\ln 2 +o(1)) \frac{\log n}{n} \approx .693 \frac{\log n}{n} , $$  where (following tradition) we assume that $\log x$ has base 2. 

 \vskip2mm 
 
 Kahn and Kalai \cite{KK1} conjectured that there is a constant $K$ such that for any \textit{increasing} family $F$, 
 $p_c (F) \le K q(F) \log n $.  This was the central question in random graph theory for many years. 
 In 2021, a  weaker version of this conjecture was proved by  Frankston, Kahn, Narayanan and 
  Park \cite{KP}, inspired by an exciting development by Alweiss, Schachar, Wu, and Zhang \cite{Lovett}. This version already contained the
 most interesting applications in the literature.  A year later, Park and Pham \cite{PP}, also using ideas from \cite{Lovett},  settled the conjecture in the following (stronger) form.

\begin{theorem}[Park-Pham \cite{PP}] \label{Theorem1} 
There is a constant $K >0$ such that for any $l$-bounded  increasing family $F, $  $p_c (F) \le K q (F) \log (l+1). $ 

\end{theorem} 

This note grew out of our attempt to teach Theorem \ref{Theorem1} in class. We found a short argument,  using induction, avoiding the relatively technical analysis in \cite{PP}.  This simplifies the proof of the main result of  \cite{PP} (Theorem \ref{main2} below)  and reduces its length 
to about one page. Other ingredients remain the same.  

Our proof also reveals that when $l \rightarrow \infty$ (which is the interesting case in applications), we can set  $K \approx  3.998$, which we believe to be the best constant with respect to the current approach.  
The previous record is $ K=8$, see  \cite{Bell}. 

\section{The covering theorem}

 In the power set $2^X$, the $m$-level, denoted by $L_m$,  is the family of all subsets of size $m$. Clearly 
$ |L_m| = {N \choose m } . $ For a family $F$ (not necessarily increasing), let  $f_p(F)$ be  the minimum expectation (with respect to $p$) of a family $G$ such that $F \subset  \langle G \rangle  $. For each given $0< p \leq 1$, we refer $f_p(F)$ as a cost function of covering $F$. A short consideration shows that if we set $p= q_F$, defined in the previous section, then $f_p (F) = f_{q_F}(F)=1/2$. 

An important property of the cost function is sub-additivity. If we partition  $F$ into $F_1 $ and $F_2$, then it is easy to see that $f_p(F) \le f_p(F_1)+ f_p(F_2) $. Furthermore, if $F$ is empty, then $f_p(F)=0$, and  if $F$ contains the empty set, then $f_p(F)=1$.  

For an increasing family $G$ (e.g. $\langle F \rangle$), let $c_t(G) = \frac{ |G \cap L_t | } { | L_t| } $ be the fraction of level $t$ in $G$. By double counting, it is easy to show that
\begin{fact} \label{doublecount}  $c_t(G)$ is increasing with $t$.  \end{fact} 

Set $m_{l}: = \lfloor LpN \log (l+1) \rfloor$ where $L$ is a sufficiently large constant. The  main result of \cite{PP} is the following 

\begin{theorem}[Park-Pham \cite{PP}]  \label{main2} There is a constant $L$ such that the following holds. 
Assume that $F$ is $l$-bounded and $p$ is chosen so that  $f_p(F) \ge 1/2 $. Then $\langle F \rangle $ contains at least $ \frac{2}{3}$ fraction of  the $m_l $-level ($L_{m_l }$).  In other words 
$$ | \langle F \rangle \cap L_{m_l}   |  \ge \frac{2}{3}   |L_{m_l } |  . $$

\end{theorem} 

As shown in \cite{PP}, Theorem \ref{Theorem1} follows quickly from Theorem \ref{main2}; see Remark \ref{remarkKK} for details. We will prove an (artificially) stronger variant, whose parameters are set for induction.

\begin{theorem}[Covering theorem]  \label{main3} There is a constant $L$ such that the following holds.  
Assume that $F$ is an $l$-bounded and $p$ is chosen so that  $f_p(F) \ge \frac{1}{2} - \frac{1}{2^{l+2}}  $. Then $\langle F \rangle $ contains at least  $\frac{2}{3} + \frac{1}{2^{l+2}}  $ fraction of  the $m_l $-level. 
\end{theorem}

\begin{proof} We prove Theorem \ref{main3}  by induction on $l$ and $N$.  For $l= 0$, $f_p(F) \ge 1/4  > 0$. This means that $F$ consists of the empty set, and the conclusion is trivial (for any $N$) 
as $\langle  \emptyset \rangle$ contains $L_m$ for all $0 \le m \le N $.  Now we prove for a pair $1 \le  l \le N$, assuming that the hypothesis holds for all values $l' < l$ and $N' \le  N$.  

In order to reduce the parameter $N$, we truncate the ground set $X$ and the family $F$ accordingly. 
For a set $W \subset X$, define $F_W := \{ S \backslash W: S \in F \} $. Let $F'_W$ be the family of minimal sets
(with respect to inclusion) of $F_W$ and let $G_W := \{T:  T \in F'_W, |T| > .9l \} $ and $\tilde F_W:= F'_W \backslash G_W$. The following lemma is the key estimation. 

\begin{lemma}[Double counting lemma] \label{DCL} For  $L > 1000$ and $w= \lfloor .1 LpN \rfloor$,  we have  $$\sum_{W: |W|= w }  f_p (  G_W )  \le   {N \choose w}   \frac{1}{8 \times 16^l } . $$

\end{lemma} 

\begin{proof}
By the definition of $f$,  $f_p (  G_W  ) \le \textbf{E}_p(G_W)=\sum_{S' \in G_W} p^{| S'| } . $ Therefore,
\begin{equation} \label{sum1} \sum_{W: |W|= w }  f_p (   G_W  )  \le \sum_{W: |W|=w} \sum_k \sum_{S': S' \in G_W , |S'| =k} p^k.  
\end{equation} 
\vskip2mm 
To bound the RHS, we bound the number of pairs $(W,S')$, in which $W$ is a set of size $w$ and $S' \in G_W$ has exactly $k$ elements. To determine $(W,S')$, we first fix the union $W'=W \cup S'$ and then choose $S' \in G_{W' \setminus S'}: |S'|=k$. There are ${N \choose {w+k}  } $ ways to choose $W'$. Once this 
union is fixed,  pick any set $S \in F$ inside the union (there must be at least one such $S$, namely, the one that defines $S'$).
For each possible $S' \subset W'$ (i.e. $W= W' \setminus S', S' \in G_{W}, |S'|=k$), if $S' \not\subset S$, then $S' \cap S = S \setminus W$ is minimal. It contradicts the fact that $S' \in G_W$ is a minimal set. Thus, $S'$ must be a subset of $S$ (regardless of the choice of $S$.) 
This simple, but delicate, point has an amazing impact, as it leads to a very good bound for the number of pairs $(W, S')$. Indeed, after choosing  $S$ (which has size at most $l$),  there are at most   ${ l \choose k }$ choices of $S'$. 

\vskip2mm 

By the definition of $G_W$,  $k > .9l$. Therefore, with $w =\lfloor .1 LpN \rfloor $, the RHS of \eqref{sum1} is at most 
\begin{equation} \label{sum2} \sum_{ .9l <k \le l}  {N \choose{w+k}} { l \choose k}  p^k  \le {N \choose w} \sum_{ .9l < k \le l } (.1 L)^{-k} { l \choose k}  \le {N \choose w} (.1 L)^{-.9l} 2^{l+1}  \le \frac {{N \choose w}}  { 8 \times 16^l }, 
\end{equation} 
given that  $L \ge  1000$. This proves the lemma. \end{proof}

Back to Theorem \ref{main3}, we say that $W$ is {\it good}  if $f _p ( G_W  )  \le \frac{1} { 2 ^{l+2} } $. By averaging, at most an $\frac{1}{2 \times 8^l} $ fraction of all $W$ are \textit{bad}. 
For a good $W$, by subadditivity, we have 
$$f_p( \tilde F_W ) \ge f_p(F'_W) - f_p( G_W) \ge f_p(F)-  \frac{1}{2 ^{l+2} }  \ge \frac{1}{2}    - \frac{1}{2^{l+1}}  \ge  \frac{1}{2} - \frac{1}{2^{l_1 +2}}, $$ with $l_1: = \lfloor .9l \rfloor < l$.

By the induction hypothesis, $ \langle \tilde F_W \rangle$ contains $\frac{2}{3} + \frac{1}{2^{l_1+2} }  $ fraction of the  $m_{l_1} $-level   of  the ground set $X\backslash W$, $|X \backslash W|=N-w$.
By taking the union with  $W$ (for  good $W$s), it follows that $\langle F \rangle$ contains at least 
$$\left(\frac{2}{3} + \frac{1}{2^{l_1 +2}} \right)\left(1-\frac{1}{2 \times 8^{l} } \right) \geq   \frac{2}{3} + \frac{1}{2^{l_1 +2}}  -  \frac{1}{2 \times 8^{l} }  \ge  \frac{2}{3}  + \frac{1}{2^{l+2} } $$  fraction of  the $(m_{l_1}  + w) $-level of $X$, in which 
$$ m_{l_1}  + w  = \lfloor L p(N-w) \log (l_1+1) \rfloor + \lfloor .1L  pN \rfloor  \le \lfloor L pN \log (l+1) \rfloor  = m_l,$$  given that $L \ge 1000$. 
Since $\langle F \rangle$ is an increasing family, by Fact \ref{doublecount}, our proof is complete. \end{proof} 

\vskip2mm 

\begin{remark} \label{remarkKK} One can easily deduce the Kahn-Kalai conjecture from the main theorem, with $K = L(1+ \epsilon)$, for any fixed $\epsilon >0$. Let us consider $G(n,p)$, the argument for random hypergraphs is similar. Set $N = {n \choose 2} $. Notice that if we choose each edge with probability $p= \rho (1+\epsilon)$ where $pN  \rightarrow \infty$,  then with probability $1-o(1)$, the resulting graph has 
at least $m= \rho N$ edges. Thus, we can generate $G(n,p)$ (barring an event of probability $o(1)$) by first generating a random number $\bar m$ of value at least $m$ (according to the binomial distribution, but this does not really matter), then hitting a uniform random point on the $\bar m$ level, and taking the corresponding graph.  Consequently, if we set $\rho=q_F$ and $\bar{m} > K q_F N \log(l+1)$, then with probability at least $2/3- o(1)$, we hit a point in $\left\langle F \right\rangle$. \end{remark}

\section{Reducing the constant $K$}\label{KK}

In this section, we show that when $l \rightarrow \infty$, we can reduce $L$ to approximately $3.998$, and then $K$ (by Remark \ref{remarkKK}) also to approximately $3.998$.  This seems to be the best value with respect to the current method.

Let $0< \delta < 1$ be a constant  and set $L$  (with foresight)  such that  $$3 > \epsilon = \frac{(L \log (1/\delta))^\delta}{2}-1 >0. $$ 
The smallest value for $L$ so that this holds for some $1 >\delta >0$ is 
$L \approx 3.998 \dots$,  which is slightly larger than the minimal value of  $ \frac{2^{1/\delta}}{ \log (1/\delta) }$  for $\delta \in (0,1)$.

Let  $l_0  $ be a natural number such that $2 \leq (1+\epsilon/3)^{l_0 -\lfloor \delta l_0 \rfloor}$. Set $m_l= \lfloor L pN \log (l+1) + 1000 pN \log (l_0+1) \rfloor. $  
With a small modification, we prove the following $\epsilon$-version of Theorem \ref{main3}.
\begin{theorem}  \label{main4} Let $l_0 \le l \le N$ be integers and $X$ be a  set of size $N$. 
Assume that $F$ is $l$-bounded and $p$ is such that $f_p(F) \ge \frac{1}{2} - (1+\epsilon/3)^{-l} $. Then $\langle F \rangle $ contains at least a $\frac{2}{3} + (1+ \epsilon/3)^{-l}$ fraction of the $m_l$-level.
\end{theorem} 
\begin{proof}

We start the induction at $l = l_0$. This base case is covered by Theorem \ref{main3} and 
results in the term $1000 pN \log (l_0+1)$.
Next, we replace the constant $.9$ (in the previous proof) by $\delta$ and consider $w:= \lfloor LcpN \rfloor$ with $c := \log (1/\delta)$. Thus, in \eqref{sum2},  instead of 
$\sum_{ .9l < k \le l } (.1 L)^{-k} { l \choose k} $, we end up with 

\begin{equation}
\begin{split}
\sum_{W: |W|= w }  f_p (  G_W  )  & \le \binom{N}{w} \sum_{ \delta l < k \le l } (Lc)^{-k} { l \choose k} \\
& \leq \binom{N}{w} (Lc)^{-l \delta} \sum_{ \delta l < k \le l } { l \choose k} \\
&  \le \binom{N}{w} (Lc)^{-l \delta} 2^l  = \binom{N}{w} \left[ \frac{(Lc)^{\delta}}{2} \right]^{-l}  \\
& = \binom{N}{w} (1 + \epsilon)^{-l} \,\,\,(\text{by Definition of $L$ and $\epsilon$ })  \\
& \le \binom{N}{w} (1 +\epsilon/3)^{-l} (1+\epsilon/3)^{-l} \,\,\, (\text{since \,  $\epsilon < 3$ }).
\end{split}
\end{equation}

We  say that $W$ is {\it good}  if $f_p (  G_W   ) \le (1+ \epsilon/3)^{-l}  $. By averaging, at most an $(1+ \epsilon/3 )^{-l}  $ fraction of all $W$ are \textit{bad}. 
For a good $W$, by subadditivity, we have 
$$f_p(\tilde{F}_W) \geq f_p(F'_W) -f_p(G_W) \geq f_p(F) - (1+\epsilon/3)^{-l} \geq \frac{1}{2}- 2(1+\epsilon/3)^{-l} \geq \frac{1}{2}- (1+\epsilon/3)^{-l_1},$$
with $l_1=\lfloor \delta l \rfloor$, thanks to the assumption  $2 \leq (1+\epsilon/3)^{l_0 -\lfloor \delta l_0 \rfloor}.$ 
By applying the induction hypothesis for $l_1$ and taking union with $W$ (for the good $W$), it follows that $\left\langle F \right\rangle$ contains at least
$$\frac{2}{3}+ (1+\epsilon/3)^{-l_1} -(1+\epsilon/3)^{-l} \geq \frac{2}{3}+(1+\epsilon/3)^{-l},$$
 fraction of the $(m_{l_1}+w)$-level for $m_{l_1}+w= \lfloor Lp(N-w)\log (l_1+1) +1000 pN \log (l_0+1) \rfloor +w$. \\
 By the settings of $l_1$ and $c$,  it is easy to check that $m_{l_1}+w$ is at most $ m_l=\lfloor LpN \log (l+1) + 1000 pN \log (l_0+1) \rfloor. $ We complete the induction and the proof. 
\end{proof}
\section{Covering theorem for arbitrary small $\epsilon_1$}
In the previous sections, we proved that if  $f_p(F) > 1/2$, then $\left\langle F \right\rangle$ contains at least $2/3 =1- \frac{1}{3}$ fraction of the $m_l$-level, for sufficiently large $m_l$. 
In \cite{Bell}, Bell considered the question of how large should  $m_l$   be if we replace $\frac{1}{3}$ by an arbitrary $\epsilon_1 >0$. He   proved in \cite[Theorem 3]{Bell} that the covering theorem still holds 
for $ m_l= \lfloor 48 p N \log l + 48p N \log \left( \frac{1}{\epsilon_1} \right) \rfloor$. By combining our  induction with Bell's result, we can prove the following bound for sufficiently small $\epsilon_1$. 

 Define $L  \approx 3.998$ and its corresponding $l_0$  as  in the beginning of Section \ref{KK}. Set $m_l= \lfloor LpN \log(l+1) + 96 pN \log \left(\frac{1}{\epsilon_1} \right)  \rfloor$.
 
\begin{theorem}\label{main5}
Let $0< \epsilon_1 <1/l_0$ be a positive number which may depend on $N$.  Assume that $F$ is $l$-bounded and $f_p(F) > \frac{1}{2}$. Then $\left\langle F \right\rangle$ contains at least $1 - \epsilon_1$ fraction of the $m_l$-level. 
\end{theorem}

\begin{proof}[Proof of Theorem \ref{main5} ] We start the induction at $l = \lfloor  \frac{1}{\epsilon_1} \rfloor-1$. This base case is covered by 
Bell's theorem and results in the term $ 96 p N \log \left( \frac{1}{\epsilon_1} \right) $. When $l \ge  \lfloor \frac{1}{\epsilon_1} \rfloor -1$, we follow the proof of Theorem \ref{main4} with $1-\epsilon_1$ replacing $\frac{2}{3}$ and obtain the remaining term $LpN \log (l+1)$. 
\end{proof}

\vskip2mm 
\noindent {\bf Acknowledgements.} This research is partially supported by Simon Foundation award SFI-MPS-SFM-00006506 and NSF grant AWD 0010308. We thank J. Park for comments and pointing out reference \cite{Bell}.



\begin{thebibliography}{99}

\bibitem{Bell}  T. Bell, The Park–Pham theorem with optimal convergence rate. \textit{Electronic Journal of Combinatorics},  (2) \textbf{30} (2023), 2-25. doi:10.37236/11600.
\bibitem{KK1} J. Kahn and G. Kalai, Thresholds and expectation thresholds, \textit{Combin. Probab. Comput}. \textbf{16} (2007), 495-502.
doi:10.4007/annals.2021.194.2.2.
\bibitem{KP} K. Frankston, J. Kahn, B. Narayanan, J. Park, Thresholds versus fractional expectation-thresholds. \textit{Ann. Mathematics} \textbf{194} (2021), 475-495. doi:10.4007/annals.2021.194.2.2.
\bibitem{Lovett} R. Alweiss, L. Schachar, K. Wu and J. Zhang, Improved bounds for the sunflower lemma, \textit{Ann. of Math}. (2) \textbf{194} (2021), no 3, 795-815. doi:10.1145/3357713.3384234.
\bibitem{PP} J. Park, H. Pham, A Proof of the Kahn-Kalai conjecture, \textit{Journal of the American Mathematical Society} (1) \textbf{37} (2024), 235-243. doi:10.1090/jams/1028.

\end{thebibliography}
\end{document}